\documentclass[12pt]{amsart}
\usepackage{amsmath,amsthm,amsfonts,amssymb,latexsym,enumerate}
\usepackage[pagebackref]{hyperref}

\newcommand{\bE}{{\mathbf E}}
\newcommand{\bC}{{\mathbf C}}

\newcommand{\bF}{{\mathbf F}}
\newcommand{\bN}{{\mathbf N}}

\newcommand{\SSS}{\mathsf{S}}

\newcommand{\Aut}{{{\operatorname{Aut}}}}

\newcommand{\Irr}{{{\operatorname{Irr}}}}
\newcommand{\Out}{{{\operatorname{Out}}}}
\newcommand{\GL}{\operatorname{GL}}

\newcommand{\SL}{\operatorname{SL}}

\newcommand{\PSL}{\operatorname{PSL}}
\newcommand{\GF}{\operatorname{GF}}

\newcommand{\Ker}{\operatorname{Ker}}

\newcommand{\CCC}{\mathsf{C}}

\newtheorem{thm}{Theorem}[section]
\newtheorem{lem}[thm]{Lemma}

\newtheorem{que}[thm]{Question}

\newtheorem*{thmA}{Theorem A}
\newtheorem*{conA'}{Conjecture A'}
\newtheorem*{thmB}{Theorem B}
\newtheorem*{conC}{Conjecture C}
\newtheorem*{conD}{Conjecture D}

\theoremstyle{definition}

\numberwithin{equation}{section}

\begin{document}

\title[Number of conjugacy classes]{Prime divisors and the number of conjugacy classes of finite groups}

\author{Thomas Michael Keller}
\address{Department of Mathematics, Texas State University, 601 University Drive, San Marcos, TX 78666, USA}
\email{keller@txstate.edu}
\author{Alexander Moret\'o}
\address{Departament de Matem\`atiques, Universitat de Val\`encia, 46100
  Burjassot, Val\`encia, Spain}
\email{alexander.moreto@uv.es}

\thanks{We thank R. Guralnick for helping us understand the results in [5]. We also thank the referee for his/her very careful reading of the paper. Parts of this work were written while both authors were attending the conference ``Counting conjectures and beyond" at the Isaac Newton Institute for Mathematical Sciences in Cambridge. We thank the institute for their financial support via the EPSRC grant no EP/R014604/1.
The research of the second author is supported by  Ministerio de Ciencia e Innovaci\'on (Grant PID2019-103854GB-I00 funded by MCIN/AEI/ 10.13039/501100011033) and Generalitat Valenciana CIAICO/2021/163.}

\keywords{Number of conjugacy classes, finite group, irreducible module}

\subjclass[2010]{Primary 20E45, Secondary 20C15, 20C20, 20D25}

\date{\today}

\begin{abstract}
We prove that there exists a universal constant $D$ such that  if $p$ is a prime divisor of the index of the Fitting subgroup of a finite group $G$, then the number of conjugacy classes of $G$ is at least $Dp/\log_2p$.  We conjecture that we can take $D=1$ and prove that for solvable groups, we can take $D=1/3$.
\end{abstract}

\maketitle



  \section{Introduction}

The study of the number of conjugacy classes $k(G)$ of a finite group $G$ has been a central theme in group theory since the work of Burnside, Frobenius and Landau around 1900. Since then, much effort has been made on the determination of finite groups with a small number of conjugacy classes (see \cite{vl1} and the references there). In a different direction, R. Brauer \cite{bra} gave the first explicit lower bound for  $k(G)$ in terms of $|G|$ and, in one of the problems in his celebrated collection of open problems, asked for substantially better bounds. This was essentially solved by L. Pyber in 1992 \cite{pyb}, but there have been a number of improvements on Pyber's bound since then (see \cite{baumeister, keller} and the references there). 

 A variation of Brauer's problem was considered by  L. H\'ethelyi and B. K\"ulshammer \cite{hk1} in 2000.  They proved that if $G$ is a finite solvable group and $p$ divides $|G|$, then  $k(G)\geq2\sqrt{p-1}$ (and this bound is best possible). This has originated a lot of, still ongoing, research. They conjectured that this inequality should hold for an arbitrary group  $G$ and even that $k(B)\geq2\sqrt{p-1}$ when $B$ is a Brauer $p$-block of $G$  of positive defect. The H\'ethelyi-K\"ulshammer bound was finally extended to arbitrary finite groups by A. Mar\'oti in \cite{mar}. Later, Mar\'oti and G. Malle proved that even the number of irreducible characters of degree not divisible by $p$ of a finite group $G$ is $|\Irr_{p'}(G)|\geq2\sqrt{p-1}$. (See \cite{hmm} for a recent strengthening of this result involving fields of values.)
  The block-version still remains open, but the principal block case has been very recently solved by N. N. Hung and A. A. Schaeffer Fry. In \cite{hs}, they prove that if $B_0$ is the principal $p$-block of a finite group $G$, then the number of irreducible characters in $B_0$ is $k(B_0)\geq2\sqrt{p-1}$. Even more, Hung, Schaeffer Fry and C. Vallejo have proven in \cite{hsv} that the number of height zero irreducible characters in the principal $p$-block is $k_0(B_0)\geq2\sqrt{p-1}$.

  On the other hand, R. Guralnick and G. Robinson \cite{gr} have proved that if $\bF(G)$ is the Fitting subgroup of a group $G$ of even order, $p$ divides $|G:\bF(G)|$ and $t\in G$ is an involution, then $p\leq|\bC_G(t)|+1$. This bound is best possible, as shown by the groups $\SL_2(p-1)$ for $p\geq5$ a Fermat prime.  We have observed that for these groups $k(\SL_2(p-1))=|\bC_G(t)|+1$, and this leads us to the question of whether $k(G)\geq p$ when $p$ is a prime that divides $|G:\bF(G)|$. This is false, as $k(\PSL_2(p))=(p+5)/2$ for any prime $p>3$. Even more, if we assume that there are infinitely many Mersenne primes (as is expected), then the best bound that we can hope for is of the order of $p/\log p$, where $\log$ (here and thoroughout the paper) means $\log_2$. In order to see this, 
  let $q$ be a prime such that $p=2^q-1$ is a Mersenne prime and consider the affine semi-linear group $G=A\Gamma(2^q)$. Then it is easy to see that $p$ divides $|G:\bF(G)|$ and
$$
k(G)=\frac{p-1}{q}+2q=\frac{p-1}{\log(p+1)}+2\log(p+1)\geq\frac{p}{\log p}.
$$
Our main result shows that indeed there exists a bound of this order of magnitude.

  \begin{thmA}
There exists a constant $D$ such that 
if $G$ is a finite  group and $p$ divides $|G:\bF(G)|$, then $$k(G)\geq D(p/\log p).$$ 
Furthermore, if $G$ is solvable, we can take  $D=1/3$.
\end{thmA}

 Note that the hypothesis $p$ divides $|G:\bF(G)|$ is redundant for non-$p$-solvable groups. 
 In this case,  we show that there are even better bounds.
 
\begin{thmB}
Let $p$ be a prime. Then
there exists $C>0$ such that if  $G$ is a non-$p$-solvable group, then $k(G)\geq Cp$.
\end{thmB}

These theorems leave open the following questions.

\begin{conC}
Let $p$ be a prime and let $G$ be a finite group. If $p$ divides $|G:\bF(G)|$, then $k(G)\geq p/\log p$.
\end{conC}

\begin{conD}
Let $p$ be a prime and let $G$ be a non-$p$-solvable  group.  Then $k(G)\geq (p+5)/2$.
\end{conD}

Note that this would be a substantial improvement on the H\'ethelyi-K\"ulshammer bound for non-$p$-solvable groups. 

We remark that the straightforward versions of the main results in this paper for the number of irreducible characters of $p'$-degree or for the number of characters in a block do not hold. For instance, the bound $|\Irr_{p'}(G)|\geq2\sqrt{p-1}$ cannot be improved even if we assume that $p$ divides $|G:\bF(G)|$. This distinguishes our results from \cite{hk1, mar} and shows that our results are tighter, in the sense that it is not possible to replace $k(G)$ by usually smaller invariants like $|\Irr_{p'}(G)|$ or $k(B)$. 
Nevertheless, we think that they could also admit a block version and a version for the number of $p'$-degree irreducible characters. This will be discussed in the last section of the paper.

\section{Preliminaries}

All groups in this paper  will be finite.
Our notation is standard and  follows \cite{isa} and \cite{manz-wolf}. If $N$ is a normal subgroup of a group $G$, then $\Irr(G|N)$ is the set of irreducible characters of $G$ whose kernel does not contain $N$. If a group $G$ acts on an abelian group  $V$, then we write $n(G,V)$ to denote the number of orbits of $G$ on $V$.

In this section, we collect some results that will be used several times in the proofs of the main theorems. We start by recalling several  results on the number of conjugacy classes that will be used without further explicit mention.

\begin{lem}
Let $G$ be a finite group and let $N\trianglelefteq G$. Then:
\begin{enumerate}
\item
$k(G)\geq k(G/N)$.
\item
$k(G)\leq k(G/N)k(N)$.
\item
If $N$ is abelian and $C=\bC_G(N)$ then $k(G)\geq k(G/C)+(|N|-1)/|G:C|$.
\end{enumerate}
\end{lem}

\begin{proof}
Parts (i) and (ii) are well-known (see, for instance, \cite{gal} for (ii)) . We prove (iii). Note that $|\Irr(G|N)|\geq (|N|-1)/|G:C|$. Hence,
$$
k(G)=|\Irr(G)|\geq|\Irr(G/C)|+|\Irr(G|N)|\geq k(G/C)+(|N|-1)/|G:C|,
$$
as wanted.
\end{proof}

The following result, which is a combination of the main results of \cite{hk2} and \cite{ms}, will allow us to consider just  groups with Sylow $p$-subgroups of order $p$.

\begin{thm}
\label{p2}
There exists a constant $c$ such that for any group $G$ whose order is divisible by the square of a prime $p$, we have $k(G)\geq cp$. Furthermore, if $G$ is solvable, then we can take $c=49/60$.
\end{thm}

We definitely cannot take $c=49/60$ when $G$ is a non-solvable $p$-solvable group, as the examples (iv) in p. 671 of \cite{hk1} show.  This is one of the reasons why our constant $D$ in Theorem A is not explicit for non-solvable $p$-solvable groups.

The following theorem of Seager \cite{sea} will be another fundamental result to get our explicit bound in the solvable case.

\begin{thm}
\label{sea}
Let $p$ be a prime and let  $G$ be a solvable primitive subgroup of $\GL_n(p)$.  If $G$ is not permutation isomorphic to a  subgroup of $\Gamma(p^n)$, then the number of orbits of $G$ on the natural module $V$ is $$n(G,V)\geq(p^{n/2}/12n)+1,$$ 
except possibly when $p^n=17^4, 19^4, 7^6, 5^8, 7^8, 13^8, 7^9, 3^6$ and $5^{16}$.
\end{thm}

Another fundamental tool will be Zsigmondy's prime theorem (see \cite[IX,8.3]{hb}). Recall that if $a>1$ and $n$ are integers, a prime $p$ is called a Zsigmondy prime divisor for $a^n-1$ if $p$ divides $a^n-1$ but $p$ does not divide $a^j-1$ for $j<n$. (This is dependent upon $a$ and $n$, and not just $a^n-1$.)

\begin{thm}
\label{zsi}
Let $a>1$ and $n$ be positive integers. Then there exists a Zsigmondy prime divisor for $a^n-1$ unless $(a,n)=(2,6)$ or $(a,n)=(2^k-1,2)$ for some positive integer $k$.
\end{thm}

We will also use several times the following consequence of the Fong-Swan theorem, that could be useful for other purposes.

\begin{lem}
\label{fs}
Let $G$ be a solvable group, and let $V$ be a faithful irreducible $G$-module over the field with $q$ elements, where $q$ is a prime. Then $|V|\geq q^d$, where $d$ is the smallest dimension of a faithful complex representation of $G$.
\end{lem}

\begin{proof}
Let $F$ be an algebraically closed field in characteristic $q$.
Write $|V|=q^n$. Let $\mathcal{X}$ be the $\GF(q)$-representation associated to $V$ and let  $\mathcal{X}^F$ be the representation  $\mathcal{X}$ viewed as an $F$-representation. 
By Theorem 9.21 of \cite{isa},  $\mathcal{X}^F$ decomposes as a sum of pairwise different irreducible $F$-representations. 
Therefore, if  $\varphi$ is the (faithful) $q$-Brauer character associated to  $\mathcal{X}^F$  we have that  $\varphi(1)=n$ and $\varphi$ is a sum of pairwise different irreducible $q$-Brauer characters. By the Fong-Swan theorem (Theorem 10.1 of \cite{nav}) all these irreducible Brauer characters can be lifted to complex irreducible characters. Therefore, $\varphi=\chi^{\circ}$ for some faithful complex character $\chi$. The result follows.
\end{proof}

We will need the  following calculus exercise together with \cite{hk1}  in the proof of the solvable case of Theorem A.

\begin{lem}
\label{cal}
If $2\leq x<1800$, then $2\sqrt{x-1}>x/(2\log x)$. If $2<x\leq 4936$ then  $2\sqrt{x-1}>0.3492x/\log x$.
\end{lem}

\section{Proof of Theorem B}

In this short section, we prove Theorem B. It is an easy consequence of results of Fulman and Guralnick on the number of conjugacy classes of simple groups of Lie type (and the classification of finite simple groups).

\begin{proof}[Proof of Theorem B]
By Theorem \ref{p2},  we may assume that $|G|_p=p$. Let $M/N=S$ be the chief factor of $G$ whose order is divisible by $p$. Arguing by induction, we may assume that $N=1$ and $\bC_G(M)=1$. Therefore, $M$ is the unique minimal normal subgroup of $G$, so $G$ is an almost simple group with socle $S$. Using the CFSG, we may assume that $S$ is of Lie type of rank $r$ over a field of size $q$, for some integers $r$ and $q$ (because there are finitely many sporadic groups, so we can ignore them and it is easy to see that if $S$ is an alternating group, then the number of conjugacy classes of $G$ is at least  the largest prime divisor of $|S|$).

Now, it follows from \cite{fg} that $k(G)\geq q^r/|M(S)|$, where $M(S)$ is the Schur multiplier of $G$. The result follows from inspection of Table 2 of \cite{her}.
 \end{proof}

\section{$p$-solvable groups}

In this section, we prove the first part of Theorem A. We start with a particular case.

\begin{lem}
\label{unique}
Let $G$ be a finite group with a unique minimal normal subgroup $N$. Suppose that $|G|_p=p$ and $N=S\times\cdots\times S$, where $S$ is a non-abelian simple $p'$-group. Then $k(G)> Cp$, where $C$ is the constant that appears in Theorem B.
\end{lem}

\begin{proof}
We know that $G$ is isomorphic to  a subgroup of $\Aut(N)=\Aut(S)\wr\SSS_n$, where $n$ is the number of direct factors isomorphic to $S$ that appear in $N$. Let $M=G\cap\Aut(S)^n$. We distinguish two cases.

Suppose first that $p$ divides $|G/M|$. Since $G/M$ is isomorphic to a subgroup of $\SSS_n$, we deduce that $n\geq p$. Hence if we fix any $1\neq x\in S$, we have that $(x,1,\dots,1), (x,x,1,\dots,1),\dots,(x,\dots,x)\in N$ are representatives of at least $p$ different conjugacy classes of $G$. The result follows in this case.

Now, we may assume that $p$ divides $|M/N|$. This implies that $p$ divides $|\Out(S)|$. Since $p$ does not divide $|S|$, it follows from the classification of finite simple groups that $S$ is a group of Lie type over the field with $q^a$ elements for some prime $q$ and some positive integer $a$ that is divisible by $p$. By the order formulae of the simple groups of Lie type (see \cite{atl}), $q^a-1$ divides $|S|$. If $l$ is a Zsigmondy prime divisor for $q^a-1$, then $q$ has order $a$ modulo $l$, so $l\equiv 1\pmod{a}$.  In particular, $l\geq a+1>p$. Applying Theorem B with the prime $l$, we deduce that there exists a constant $C$ such that $k(G)\geq Cl>Cp$, as wanted.
\end{proof}



Next, we handle  the critical case.

 \begin{thm}\label{p-solvable} There exists a universal constant $A$ such that the following holds.
Let $p$ be a prime and $G$ a  group such that the Sylow $p$-subgroup of $G$ 
is normal in $G$ and cyclic of order $p$. Let $V$ be a faithful $G$-module over $GF(q)$ where
$q$ is a prime different from $p$.  Then
\[k(GV) \geq   A\ \frac{p}{\log p}.\]
\end{thm}

\begin{proof} We first define a constant $A$ that will do the job.
Let $\beta$ be as in \cite[Theorem B]{keller}. Moreover let $\delta_0=1/\delta$, where $\delta$ is the constant
one gets from \cite[Theorem 2.1]{baumeister} if one chooses $\epsilon =1$ there. Now choose $A>0$ sufficiently small such
that it satisfies the following conditions:\\
\[  (i)\ \ \ A\leq \frac{\beta}{2}\ \frac{p-1}{p}; \ \ \ \ (ii)\ \ \  \log(1/A)>e^4, \ \mbox{and}\ \ (iii)\ \ \ \frac{\log(1/A)}{(\log\log (1/A))^4}\geq \frac{8\delta_0}{\beta}.\]

First note that there is nothing to prove if $A p/\log  p\leq 1$, so we may assume that $p/\log p\geq 1/A$
and thus $p>1/A$. Since the function $f(x)=x/(\log x)^4$ is strictly increasing for $x>e^4$, we see that with $(ii)$ and $(iii)$ we
obtain
\[\frac{\log(p)}{(\log\log (p))^4}\geq \frac{8\delta_0}{\beta}.\]
Hence for any real number $a_0$ with $a_0\geq \log(p)$ we get the inequality
\[ \frac{\log(a_0)}{(\log\log (a_0))^4}\geq \frac{8\delta_0}{\beta},\]
which can easily be seen to be equivalent to
\[ (*)\ \ \  \beta a_0 -16\delta_0(\log a_0)^4\geq (\beta/2) a_0.\]

Now let $GV$ be a counterexample of minimal order. 
 Let $P$ be the Sylow $p$-subgroup of $G$ and observe that since $P$ is normal and cyclic we see that 
$V_P=[P,V]\oplus C_V(P)$, and $P$ acts frobeniusly on $[P,V]$. By induction we may assume that $V=[P,V]$,
and so $P$ acts frobeniusly on $V$. Now if $0<V_1<V$ is a $G$-module, then $G$ also acts on $V/V_1$ and $P$ acts
frobeniusly on $V/V_1$, and so by induction we are done again. Hence no such $V_1$ exists, and hence $V$ is irreducible
as $G$-module.

Write $M=\bC_G(P)$ and note that $M=K\times P$ for a Hall $p'$-subgroup $K$ of $M$. Then $G/M$ is isomorphic to a subgroup of 
Aut($P$), and hence if we write $a=|G/M|$, then $a$ divides $p-1$, and $G/M$ is cyclic of order $a$. Now we see that 
\[k(GV)\geq k(G)\geq a + \frac{p-1}{a}.\ \ \ \ (0)\]

Also, if $g\in G$ is such that $G/M=\langle gM\rangle$, then $P\langle g\rangle$ is a subgroup $U$ of $G$ which has a factor group that is a 
Frobenius group of order $ap$. So if $0<V_0<V$ is an irreducible  $U$-submodule of $V$, then by Lemma \ref{fs} we have $|V|\geq |V_0|\geq q^a\geq 2^a$.

Now let $S$ be the solvable radical of $G$. Then $SV$ is a solvable group with trivial Frattini subgroup, and hence by \cite[Theorem B]{keller}
with $\beta$ as introduced above we have that
\[k(SV)\geq |SV|^{\beta}\geq |V|^{\beta}\geq 2^{a\beta}. \ \ \ \ (1)\]

Furthermore, $G/S$ is a group with trivial solvable radical, and so by \cite[Theorem 2.1]{baumeister} with   $\delta_0>0$
as introduced above we get 
\[\log |G/S| \leq \delta_0(\log k(G/S))^4.\]

Furthermore, note that $k(K)\leq a$, because if $k(K)>a$, then $$k(M)=k(P\times K)=pk(K)> pa$$ and then
$$k(GV)\geq k(G)\geq k(M)/|G/M|> pa/a=p,$$ contradicting $GV$ being a counterexample. Now since $P\leq S$,
we see that $$MS/S\cong M/(M\cap S)\cong K/(K\cap S)$$ and hence $k(MS/S)\leq k(K)\leq a$. Since $|G:M|=a$, we then conclude
that $$k(G/S)\leq k(G/MS)k(MS/S)\leq a\cdot a=a^2.$$ With this we get
\[ |G/S|\leq  2^{ \delta_0(\log a^2)^4}= 2^{ 16\delta_0(\log a)^4}. \ \ \ \ (2)\]

Putting (1) and (2) together, we obtain
\[p\geq k(GV) \geq \frac{k(SV)}{|G/S|}\geq \frac{2^{a\beta}}{2^{ 16\delta_0(\log a)^4}}\ \ \ \ \ (3).\]
Note 
that by (0) we may assume that $a\geq \log p$.
Therefore by $(*)$ we know that $\beta a -16\delta_0(\log a)^4\geq (\beta/2) a$.
Therefore by (3) we get
\[\log p\geq \beta a -16\delta_0(\log a)^4\geq (\beta/2) a.\]
Hence $a\leq (2/\beta)\log p$, and thus (0) yields 
\[k(GV)\geq \frac{p-1}{a}\geq \frac{p-1}{ (2/\beta)\log p}=\frac{\beta}{2} \frac{p-1}{\log p}.\]
But now by $(i)$  we see that
\[\frac{\beta}{2} \frac{p-1}{\log p}\geq A \frac{p}{\log p}\]
which is the final contradiction to $GV$ being a counterexample. The proof is complete.
\end{proof}

Now, we can complete the proof of the first part of Theorem A. We refer the reader to \cite[9A]{isafg} for the definition and basic properties of the generalized Fitting subgroup.

\begin{thm}
\label{psol}
There exists a constant $C$ such that if $G$ is a finite  group  such that $p$ divides $|G:\bF(G)|$ then $k(G)\geq Cp/\log p$.
\end{thm}

\begin{proof}
By Theorem B, we may assume that $G$ is $p$-solvable. 
By Theorem \ref{p2}, we may assume that $|G|_p=p$. We argue by induction on $|G|$. By III.4.2(d) of \cite{hup}, we know that $\bF(G/\Phi(G))=\bF(G)/\Phi(G)$. Hence, $p$ divides $|G/\Phi(G):\bF(G/\Phi(G))|$. Since $k(G)\geq k(G/\Phi(G))$, we may assume that $\Phi(G)=1$. Note that this implies that the Frattini subgroup of any subnormal subgroup of $G$ is trivial and in particular any component of $G$ is simple. Hence $\bE(G)$ is a direct product of (non-abelian) simple $p'$-groups. Furthermore, $\bF(G)$ is a direct product of elementary abelian minimal normal subgroups of $G$, by III.4.5 of \cite{hup}.

Hence the generalized Fitting subgroup $\bF^*(G)=\bE(G)\bF(G)$ is the direct product of the minimal normal subgroups of $G$, i.e., coincides with the socle of $G$. Note also that $p$ does not divide $|\bF^*(G)|$. Write $\bF^*(G)=V_1\times\cdots\times V_t$, where $V_1,\dots, V_t$ are minimal normal subgroups of $G$.

 Since $\bC_G(\bF^*(G))\leq\bF^*(G)$, we deduce that $\bC_G(\bF^*(G))=\bF(G)$. Write $C_i=\bC_G(V_i)$ for every $i$. Note that $\bC_G(\bF^*(G))=\bigcap_{i=1}^t C_i$ so $G/\bF(G)$ is isomorphic to a subgroup of the direct product  $G/C_1\times\cdots\times G/C_t$. Since $p$ divides $|G/\bF(G)|$, there exists $j$ such that $p$ divides $|G/C_j|$. Write $V=V_j$ and $C=C_j$. Note that since $p$ divides $|G/C|$, $p$ does not divide $|C|$.

 Suppose first that $V$ is abelian.
Using III.4.4 of \cite{hup}, we deduce that there exists $H\leq G$ such that $G=HV$ and $H\cap V=1$. Let $D=\bC_H(V)$, which is a normal $p'$-subgroup of $G$. Hence $VD/D$ is the unique minimal normal subgroup of $G/D$ and $H/D$ acts faithfully on $VD/D$. Since $k(G)\geq k(G/D)$ and the hypotheses hold for $G/D$, we may assume that $D=1$. Hence, $V$ is the unique minimal normal subgroup of $G$ and $\bF(G)=V$. By induction, again, we may assume that $p$ divides $|\bF_2(G):\bF(G)|$, so $H$ has a normal Sylow $p$-subgroup. Now, the result follows from Theorem \ref{p-solvable}.

Hence, we may assume that $V$ is not abelian. In this case, $VC/C$ is the unique minimal normal subgroup of $G/C$ and we are in the situation of Lemma \ref{unique}. The theorem follows.
\end{proof}

\section{Solvable groups: imprimitive case}

In this and the next section, we work to provide the explicit and reasonable constant in the solvable case of Theorem A.
The problem is reduced to studying the action of a group on a finite faithful irreducible module. Our first result handles the imprimitive case.

 \begin{thm}\label{prop} There exists a universal constant $D$ such that the following holds.
Let $p$ be a prime and $H$ a solvable group such that the Sylow $p$-subgroup of $H$ 
is normal in $H$ and cyclic of order $p$. Let $V$ be a faithful irreducible $H$-module over $GF(q)$ where
$q$ is a prime different from $p$, and suppose that $V$ is not quasi-primitive (in particular,
it is imprimitive).  Then
\[k(HV) \geq   D\ \frac{p}{\log p}.\]
One can choose $D=0.3492$. Moreover, if $p>11000$, then one can even choose $D=1/2$.
\end{thm}

\begin{proof}
Let $0<D\leq 1/2$ be arbitrary. We will see that the proof will work with this $D$ except for one critical case.
In that one case the argument will only work under the additional assumption $p>11000$, and for arbitrary $p$ it will only
work if $D\leq 0.3492$.\\

First note that by Lemma \ref{cal} and \cite{hk1}, we may assume that $2\sqrt{p-1}\leq Dp/\log_2(p)$, that is, $D \geq (2\sqrt{p-1}\log_2(p))/p$.
If $D=1/2$, this forces $p>1800$, and if $D=0.34924$,  this forces $p>4936$.\\
Let $HV$ be a counterexample of minimal order. 


Now let $N\unlhd H$ be maximal such that $V_N$ is not homogeneous, and then by Clifford's theorem we
know that $V_N=\oplus_{i=1}^n V_i$ for some $n>1$ and homogeneous components $V_i$, and $H/N$ permutes the $V_i$ ($i=1,\dots ,n$)
faithfully and primitively.

Next, let $P$ be the Sylow $p$-subgroup of $H$ and observe that since $P$ is normal and cyclic, it acts frobeniusly on $V$. We claim that $P\leq N$. If not, then $P$ would permute the
$V_i$ nontrivially, and we could without loss assume that $\{V_1, \dots , V_p\}$ would be an orbit of $P$ in this action. But then if $x$ is a
generator of $P$ and $0\ne v_1\in V_1$, then $\sum\limits_{i=0}^{p-1}v_1^{x^i}$ would clearly be a nonzero element of $V$ which is
fixed by $x$ and thus $P$, contradicting $P$ acting frobeniusly on $V$. This proves our claim that $P\leq N$.

We study the action of $N$ on $V$. We claim that $N$ has at least $p^{n-1}/(p-1)$ distinct orbits on $V$   (*). \\

To see this, we let $0\ne v_1\in V_1$.
Let $g_i\in H$ ($i=1, \dots , n$) such that $V_1^{g_i}=V_i$ for all $i$; clearly we may assume that $g_1=1$. Put $v_i=v_1^{g_i}\in V_i$ for
all $i$. Then
\[T:=\sum\limits_{i=1}^n v_i^N\]
obviously is an $N$-invariant subset of $V$. For $j=1,\dots n$ we define
\[T_j:=\sum\limits_{i=1}^j v_i^N\]
so that $T_n=T$. Write $M=\bC_N(P)$ and note that $M=K\times P$ for a Hall $p'$-subgroup $K$ of $M$.
To establish our claim (*), we prove the following. 

(**) Let $j\in\{1,\dots , n\}$. Then $M$ has at least $p^{j-1}$ orbits on $T_j$.

We prove this by induction on $j$. If $j=1$, then the assertion is trivial. So let $1<j\leq n$  and suppose that $M$ has at least
$p^{j-2}$ orbits on $T_{j-1}$. Now let $y\in T_{j-1}$ be in one of these orbits. Since $P$ is normal in $H$ and acts frobeniusly on $V$,
we conclude that $p$ does not divide the order of $\bC_M(y)$; that is, $\bC_M(y)=\bC_K(y)$. Since $M$ is $p$-nilpotent and $P$ acts frobeniusly 
on $V$, the $N$-orbit $v_j^N\subseteq V_j$ decomposes into at least $p$ distinct $\bC_M(y)$-orbits (since the sizes of the $\bC_M(y)$-orbits
are not divisible by $p$, but $|v_j^M|$ is divisible by $p$, and $P$ will always join $p$ of the $\bC_M(y)$-orbits to a single $\bC_M(y)\times P$-orbit).
If $w_i\in v_j^N$ ($i=1, \dots , p$) are representatives of such $p$ distinct $\bC_M(y)$-orbits, then we see that the
$y+w_i\in T_j$ ($i=1, \dots , p$) are representatives of $p$ distinct $M$-orbits. Since $y$ was chosen from an arbitrary $M$-orbit on
$T_{j-1}$, and since furthermore clearly any $y_1+w_i$ and $y_2+w_j$ (for any $i,j\in\{1,\dots ,p\}$ lie in different $M$-orbits if
$y_1$ and  $y_2$ lie in different $M$-orbits on $T_{j-1}$, it follows that $M$ has at least $p^{j-2}\cdot p=p^{j-1}$ distinct orbits on $T_j$.
Thus (**) is proved.

Now $|N/M|=|\bN_N(P)/\bC_N(P)|$ divides $p-1$, and so by (**) for $j=n$ it follows that $N$ has at least $p^{n-1}/(p-1)$ orbits on $V$,
establishing (*).

Now by a result of Dixon \cite{dixon} we know that $|H/N|\leq 24^{(n-1)/3}$, and since $H/N$ permutes the $N$-orbits on $V$, and the $H$-orbits
on $V$ are clearly conjugacy classes of $HV$, with (*) we obtain
\[k(HV)\geq \frac{n(N,V)}{|H/N|}\geq \frac{    \frac{p^{n-1}}{p-1}   }{24^{(n-1)/3}}=\frac{1}{p-1}\
\left(\frac{p}{2\sqrt[3]{3}}\right)^{n-1}.\ \ \ \ \  (+)\]
This implies that if
\[ p^{n-3}\ \geq\ 2^{n-1}\ \cdot 3^{\frac{n-1}{3}},\ \ \ \ (***)\]
then $k(HV)\geq p$, contradiction, and we are done.

Hence as $p>1000$, $p$ will satisfy (***) unless
$n\leq 3$. Since we are done if (***) holds, we may assume that $n\leq 3$, i.e., $n=2$ or $n=3$.

First suppose that $n=3$. Then by $(+)$ we get that $k(HV)\geq \frac{p}{4\sqrt[3]{9}}\geq \frac{Dp}{\log_2(p)}$.
So we are done when $n=3$.

So let finally $n=2$. Then $|H/N|=2$. Recall that $M=K\times P$ for a Hall $p'$-subgroup $K$ of $M$. 

We first want to show that  $M$ acts irreducibly on $V_1$. To do so, assume that $V_1=X_1\oplus X_2$ for
nontrivial $M$-modules $X_1$ and $X_2$. Then we also have $V_2=Y_1\oplus Y_2$ for
nontrivial $M$-modules $Y_1$ and $Y_2$. Arguing just as in the proof of (**), we let $0\ne x_1\in X_1$,
so that each $M$-orbit on $X_2$ splits in at least $p$ nontrivial $\bC_M(x_1)$-orbits. If $x_2$ is in one of these orbits,
then  each $M$-orbit on $Y_1$ splits in at least $p$ nontrivial $\bC_M(x_1+x_2)$-orbits.  If $y_1$ is in one of these orbits,
then  each $M$-orbit on $Y_2$ splits in at least $p$ nontrivial $\bC_M(x_1+x_2+y_1)$-orbits. This shows that
$M$ has at least $p^3$ nontrivial orbits on $V$. Thus
\begin{eqnarray*}
k(HV)\ &\geq\ & \frac{k(MV)}{|HV:MV|}\ \geq \  \frac{n(M,V)}{|H:M|}\ \\
&\geq\  &
\frac{p^3}{2(p-1)}\ \geq\  p,\\
\end{eqnarray*}
contradicting $HV$ being a counterexample. So indeed $M$ acts irreducibly on $V_1$ (and also on $V_2$).

Next we claim that $M$ acts primitively on $V_1$. 
To see this, assume to the contrary that $V_1=W_1\oplus \dots\oplus W_l$
is the sum of $l\geq 2$ subspaces $W_i$ ($i=1,\dots l$) which are transitively permuted by $M$. Let $L\unlhd M$
be the kernel of this permutation action.
Then likewise
$V_2=W_{l+1}\oplus \dots\oplus W_{2l}$
is the sum of $l$ subspaces $W_i$ ($i=l+1,\dots ,2l$) which are transitively permuted by $M/L$.
Now if $P$ is not a subgroup of $L$, then $P$ permutes the $W_i$ with orbits of length $p$. Hence $l\geq p$ and
we may assume that $\{W_1,\dots , W_p\}$ is a $P$-orbit. But then taking $0\ne w_1\in W_1$ and adding up the
elements in $w_1^P$ we easily find a nonzero fixed point of $P$ on $W_1\oplus\dots\oplus W_p$, contradicting $P$
acting frobeniusly on $V$. Thus indeed $P\leq L$.\\
We can now argue as for $M$ above to see that $L$ has at least $p^{2l-1}$ nontrivial orbits on $V$. Using Dixon's result
again, we know that $|M/L|\leq 24^{(2l-1)/3}\leq 3^{2l-1}$ so that we get
\begin{eqnarray*}
k(HV)\ &\geq\ & \frac{k(LV)}{|HV:LV|}\ \geq \  \frac{n(L,V)}{|H:L|}\ \\
&\geq\  &
\frac{p^{2l-1}}{|H:N| |N:M| |M:L|}\ \geq\  \frac{p^{2l-1}}{2(p-1) 3^{2l-1}},\\
\end{eqnarray*}
where the first inequality follows from the well-known fact that  the number of conjugacy classes in a subgroup is bounded
by its index times the number of conjugacy classes in the whole group \cite{gal}. Now since $p>1000$ and $l\geq 2$, we further get
\begin{eqnarray*}
k(HV)\ &\geq\ &  \frac{p^{3}}{2(p-1) 3^{3}}\geq\   \frac{p^{2}}{54}   \geq p ,
\end{eqnarray*}
contradicting $HV$ being a counterexample. This shows that $V_1$ (and thus also $V_2$) is primitive as an $M$-module.
Then clearly $V$ is also primitive as an $N$-module.

We keep writing $|V_1|=q^m$ and apply Seager's result (Theorem \ref{sea}) to the action of $N/\bC_N(V_1)$ on $V_1$.
This yields that we have one of the following situations:

\noindent
(1) $N/\bC_N(V_1)$ is permutation isomorphic to a subgroup of $\Gamma(q^m)$;\\
(2) If $r$ is the number of orbits of $N$ on $V_1$, then $r>(q^{m/2}/12m) +1$; or\\
(3) $q^m\in\{17^4, 19^4, 7^6, 5^8, 7^8, 13^8, 7^9, 3^{16}, 5^{16}\}$.

Let $W\leq V$ be an irreducible $P$-module. Since $V_P$ is homogenous,
 then there is a $t\in\mathbb{N}$ such that $ V_P$ is a direct sum of $t$ copies of $W$.
Now we have $q^m=|V|=|W|^t$, so if we write $|W|=q^s$, then we see that $m=st$. 

Note that if $s=1$, then by \cite[Theorem 2.1]{manz-wolf} we have that (1) holds; so if (1) does not hold, then
we know that $s\geq 2$. 

We first assume that (2) holds, but (1) does not hold. As just observed, then $s\geq 2$.
Moreover,  from (2) we get
\begin{eqnarray*}
r\ &>\ &  \frac{q^{m/2}}{12m}+1\ =\   \frac{|W|^{t/2}}{12st}+1.
\end{eqnarray*}

Now if $v_1\in V_1$, then $\bC_N(v_1)\leq N$ also has at least $r$ distinct orbits on $V_2$, and this shows that each orbit of $N$
on $V_1$ gives rise to at least $r$ orbits of $N$ on $V$ which all have a representative that has $v_1$ as their $V_1$-component.
Then we have $n(N,V)\geq r^2$, and with the above and the
fact that $s=\log_q|W|$ we see that
\[n(N,V)\ \geq\ r^2\ \geq\  \frac{|W|^{t}}{144t^2(\log_q|W|)^2}.   \]
Now since $|W|>|P|=p>1000$, we see that the function $f(x)= |W|^x/x^2$ is increasing for $x\geq 2$, and hence
$|W|^t/t^2\geq |W|^2/2^2=|W|2t/4$, and thus
\begin{eqnarray*}
n(N,V)\ &\geq\ &\frac{|W|^{2}}{576(\log_q|W|)^2} \ \geq\ \frac{|W|^{2}}{576(\log_2|W|)^2}\ \\
\ &=\ & \frac{D|W|}{(\log_2|W|)}\cdot \frac{|W|}{576D\log_2|W|}\ \geq\ \frac{Dp}{\log_2p}\cdot \frac{p}{576D\log_2p}.
\end{eqnarray*}
Since clearly $k(HV)\geq n(N,V)/|H/N|=n(N,V)/2$, we see that we are done if $\frac{p}{576D(\log_2p)}\geq 2$. This happens if and
only if \[\frac{p}{\log_2p}\geq 1152 D.\] 
Now first note that if $p>11000$ and $D=1/2$, then this is the case (we actually could choose $D$
slightly larger here) and we are done. \\
In general, however, with our lower bound on $D$ from the beginning of the proof we obtain
\[(2\sqrt{p-1}\log(p))/p \leq D \leq p/(1152\log(p)) \]
This is the critical case mentioned at the beginning of the proof and is thus the one which determines $D$.
The optimal $D$ is reached when we have equality here, so we have to find $p$ such that
the left hand side is equal to the right hand side. With the help of a calculator one easily finds that 
 this happens exactly when $p = 4936.0274$,
and then the value of $D$ equals 0.3492276. Therefore if $D=0.3942$, then we are done in this situation as well.\\


Next we assume that (3) holds, but (1) does not hold. As observed above, since (1) does not hold, we have $s\geq 2$.
Hence $|W|\leq q^{m/2}$ and actually $|W|= q^{m/s}$. Since $p=|P|$ divides $|W|-1$ and since (3) holds
(and since $a-1$ divides $a^b-1$ for any positive integers $a$, $b$), we see that $p$ is
a prime divisor of one of the following numbers: $17^2-1, 19^2-1, 7^2-1, 7^3-1, 5^4-1, 5^2-1, 7^4-1, 7^2-1, 11^4-1, 11^2-1,
13^4-1, 7^3-1,  3^8-1, 3^4-1, 3^2-1, 5^8-1, 5^4-1, 5^2-1$. But recall from the very beginning that $p>1000$, and none of the numbers
listed has a prime divisor this large. This is a contradiction.\\

It remains to consider the case that (1) holds. Then $N$ is isomorphic to a subgroup of $\Gamma(q^m)\times\Gamma(q^m)$.
Recall that $\Gamma(q^m)$ is of order $(q^m-1)m$ and has a normal cyclic subgroup $\Gamma_0(q^m)$. Write
$N_0=N\cap (\Gamma_0(q^m)\times\Gamma_0(q^m))$, so that $N_0/\bC_{N_0}(V_i)$ acts frobeniusly on $V_i$ ($i=1,2$).
Furthermore, we have $|N/N_0|\leq m^2$.\\

Let $s$ be as above, that is, $V_1$ viewed as a $P$-module, is the direct sum of $s$ copies of an irreducible $P$-module $W$.
First suppose that $s\geq 2$. Writing again $n(N,V)$ for the number of orbits of $N$ on $V$, we have
\[k(NV)\geq n(N,V)\geq \frac{|V|}{|N|}\geq \frac{q^{2m}}{|N_0|m^2}=\frac{1}{|N_0|}\left(\frac{q^m}{m}\right)^2.\]
Since $k(HV)\geq k(NV)/2$, then we are done if $\frac{1}{|N_0|}\left(\frac{q^m}{m}\right)^2\geq\frac{2Dp}{\log_2p}$.
Therefore we  now assume that $\frac{1}{|N_0|}\left(\frac{q^m}{m}\right)^2 <\frac{2Dp}{\log_2p}$ or, equivalently,
$|N_0|>(1/2D)((\log_2 p)/p)(q^m/m)^2$.\\
Now since $p$ divides $|W|-1$ and $s\geq 2$, it follows that $p\leq |V_1|^{1/2}=q^{m/2}$ and $p^2\leq q^m$. Therefore
 we argue as follows.

\begin{eqnarray*}
k(HV) &\geq & \frac{k(NV)}{2}\ \geq   \frac{k(N)}{2}\geq   \frac{\frac{k(N_0)}{|N/N_0|}}{2}
 = \frac{|N_0|}{2|N/N_0|}\geq \frac{|N_0|}{2m^2}\\
&\geq &\frac{(1/2D)((\log_2 p)/p)(q^m/m)^2}{2m^2}
 =  \frac{1}{4D}\frac{\log_2 p}{p}\left(\frac{q^m}{m^2}\right)^2
\\
& \geq & \frac{1}{4D}\frac{\log_2 p}{p}\left(\frac{q^m}{(\log_2 q^m)^2}\right)^2\geq
\frac{1}{4D}\frac{\log_2 p}{p}\left(\frac{p^2}{(\log_2 p^2)^2}\right)^2\\
&\geq & \frac{1}{64D}\frac{\log_2 p}{p}\left(\frac{p}{\log_2 p}\right)^4=\frac{1}{64D}\left(\frac{p}{\log_2 p}\right)^3
= \frac{Dp}{\log_2 p}\ \cdot\ \frac{1}{64D^2}\left(\frac{p}{\log_2 p}\right)^2.
\end{eqnarray*}
Now since $p>1000$, it is easy to see that the last factor $\frac{1}{64D^2}\left(\frac{p}{\log_2 p}\right)^2$ is greater than 1,
and so we have the desired contradiction to $HV$ being a counterexample and are done.

So let finally $s=1$. We argue somewhat similarly to the case that $s\geq 2$, but need some more detailed structural information
first. Since $s=1$, we know that $V_1$ is irreducible as a $P$-module. By \cite[Example 2.7]{manz-wolf} we know that $p$ is
a Zsigmondy prime divisor of $q^m-1$. Since $P$ acts faithfully (actually, frobeniusly) on $V_1$ and $V_2$, with
\cite[Lemma 6.5(c)]{manz-wolf} we conclude that $\bC_N(P)=N_0$. In particular, $N/N_0$ is isomorphic to a subgroup
of $\Aut(P)$, and thus $N/N_0$ is cyclic.
On the other hand, since $N$ is isomorphic to a subgroup of $\Gamma(q^m)\times\Gamma(q^m)$, we see that $N/N_0$ is
isomorphic to a subgroup of a homocyclic group of type $(m,m)$, so in particular the exponent of $N/N_0$ is bounded above by
$m$. Altogether it follows that $|N/N_0|\leq m$.

Now
\[k(NV)\geq n(N,V)\geq \frac{|V|}{|N|}\geq \frac{q^{2m}}{|N_0|m}.\]
Since $k(HV)\geq k(NV)/2$, then we are done if $\frac{1}{|N_0|}\frac{q^{2m}}{m}\geq\frac{2Dp}{\log_2p}$.
Therefore we  now assume that $\frac{1}{|N_0|}\frac{q^{2m}}{m} <\frac{2Dp}{\log_2p}$ or, equivalently,
$|N_0|>(1/2D)((\log_2 p)/p)(q^{2m}/m)$.\\
As clearly $p<|V_1|=q^m$, we can argue as follows.

\begin{eqnarray*}
k(HV) &\geq & \frac{k(NV)}{2}\ \geq   \frac{k(N)}{2}\geq   \frac{\frac{k(N_0)}{|N/N_0|}}{2}
 = \frac{|N_0|}{2|N/N_0|}\geq \frac{|N_0|}{2m}\\
&\geq &\frac{(1/2D)((\log_2 p)/p)(q^{2m}/m)}{2m}
 \geq  \frac{1}{4D}\frac{\log_2 p}{p}\left(\frac{q^m}{m}\right)^2
 \geq  \frac{1}{4D}\frac{\log_2 p}{p}\left(\frac{q^m}{\log_2 q^m}\right)^2\\
&\geq &\frac{1}{4D}\frac{\log_2 p}{p}\left(\frac{p}{\log_2 p}\right)^2 = \frac{1}{4D}\frac{p}{\log_2 p}
=\frac{Dp}{\log_2 p}\cdot \frac{1}{4D^2}.
\end{eqnarray*}
Hence we are done if $\frac{1}{4D^2}\geq 1$ or, equivalently, $D\leq 1/2$. Fortunately, this is the case, and the
proof is complete.
\end{proof}



\section{Solvable groups: primitive case}

Next, we consider  the primitive case.

 \begin{thm}\label{prop2} 
Let $p$ be a prime and $H$ a solvable group such that the Sylow $p$-subgroup of $H$ 
is normal in $H$ and cyclic of order $p$. Let $V$ be a faithful irreducible $H$-module over $GF(q)$ where
$q$ is a prime different from $p$, and suppose that $V$ is primitive. Then
\[k(HV) \geq   \ \frac{p}{2\log p}.\]
\end{thm}

\begin{proof}
By Lemma \ref{cal} and \cite{hk1}, we may assume that $p>1800$.
Let $P$ be the normal Sylow $p$-subgroup of $H$.
Let $C=\bC_H(P)$ and put $n=|H:C|$.  Since $H/C$ is isomorphic to a subgroup of $\Aut(P)$, we deduce that $H/C$ is cyclic and $n$ divides $p-1$. If $\lambda\in\Irr(P)$ is nonprincipal, then $I_H(\lambda)=C$. By Clifford's correspondence (Theorem 6.11 of \cite{isa}), any faithful  complex representation of $H$ has dimension at least $|H:I_H(\lambda)|=n$.  Furthermore, note that $C=P\times U$, where $U$ is a Hall $p'$-complement of $C$. By Gallagher's theorem (Corollary 6.17 of \cite{isa}) and Clifford's correspondence, there are $k(U)$ irreducible characters of $H$ lying over $\lambda$. We conclude that $H$ has exactly $(p-1)k(U)/n$ irreducible characters lying over nonprincipal irreducible characters of $P$.
Since $H/C$ is abelian, we conclude that
\begin{align}
\label{kh}
k(H)\geq n+\frac{(p-1)k(U)}{n}.
\end{align}
In particular, the result follows if either $n\leq2\log p$ or $n\geq p/2\log p$, so we may assume that $2\log p<n<p/2\log p$.

Since $V$ is a faithful irreducible module in characteristic $q$, using Lemma \ref{fs} we deduce that 
$$|V|\geq q^n\geq2^n>p^2.$$

Suppose that $2\log p<n\leq 4\log p$.  Then
$$
k(H)\geq 2\log p+\frac{(p-1)k(U)}{4\log p}.
$$
In particular  $k(H)\geq p/4\log p$.
Also, the result follows if $k(U)\geq 2$, so we may assume that $U=1$.  Thus $$|H|=pn\leq4p\log p,$$ so $$n(H,V)\geq p^2/4p\log p=p/4\log p.$$
Since $$k(HV)\geq k(H)+n(H,V)\geq\frac{p}{4\log p}+ \frac{p}{4\log p}=\frac{p}{2\log p},$$ the result follows in this case too. 

Hence, we may assume that $4\log p<n<p/2\log p$. 
Recall that by hypothesis $V$ is a primitive $H$-module. (Note that we have not used this hypothesis in the remaining cases.)  If $H$ is not isomorphic to a subgroup of $\Gamma(V)$, then Theorem \ref{sea} implies that the number of nontrivial orbits of $H$ on $V$ is $n(H,V)\geq q^{n/2}/12n$ except for when $|V|=17^4, 19^4, 7^6, 5^8, 7^8, 13^8, 7^9, 3^6$ or $5^{16}$. Thus, in the nonexceptional cases, we have
$$
k(HV)\geq k(H)+n(H,V)\geq n+\frac{p-1}{n}+\frac{2^{n/2}}{12n}.
$$
Note that the third summand is  $\geq p^2/12(p/2\log p)\geq p\log p/6>p/2\log p$ since $p>1800$. The result also follows.

Now, we consider the exceptional cases.  In all of them $|V|\leq q^{16}$. On the other hand, we know that $|V|\geq q^n$. We conclude that $n\leq 16$. Using (\ref{kh}) and the fact that  $p> 1800$ , we have that
$$
k(H)\geq \frac{p-1}{n}\geq\frac{p-1}{16}\geq\frac{p}{2\log p},
$$
and the result follows.

Finally, assume  that $H$ is isomorphic to a subgroup of $\Gamma(V)=CN$ with $N=\GF(q^m)^{\times}$ and $C$ is the Galois group of $\GF(q^m)/\GF(q)$. Thus $D=N\cap H$ is a normal subgroup of $H$ of order a divisor of $q^m-1$. Write $|D|=pf$. Note that $H/D$ is cyclic of order $e$, for some divisor $e$ of $m$. As before,
$$
k(H)\geq e+\frac{pf}{e}.
$$
Thus, we may assume that $e\leq p/(2\log p)$.
If $e\leq 2f\log p$, then the second summand is $\geq p/(2\log p)$ and we are done. Since $2f\log p<e\leq p/(2\log p)$, we deduce that $f<p/(2\log p)^2$.  It follows that
$$
|H|=e\cdot p\cdot f\leq\frac{p}{2\log p}\cdot p\cdot \frac{p}{(2\log p)^2}=\frac{p^3}{(2\log p)^3}.
$$
On the other hand,  since $m\geq e\geq n>4\log p$,
 $$|V|>q^{4\log p}\geq p^4.$$
Hence, the number of orbits of $H$ on $V$ is at least $p^{4}/|H|>p$, and we are done in this case too. This completes the proof.
\end{proof}

The following  is a slightly strengthened version of the second part of Theorem A.

\begin{thm}
\label{11000}
If $G$ be a solvable  group and $p$ divides $|G:\bF(G)|$, then $$k(G)\geq 0.3492(p/\log p).$$ 
Furthermore, if $p>11000$ then $$k(G)>\frac{p}{2\log p}.$$
\end{thm}

\begin{proof}
If $p^2$ divides $|G|$, it follows from  Theorem \ref{p2}, that $k(G)\geq (49p+1)/60\geq p/(2\log p)$. Hence, we may assume that $|G|_p=p$. Since $k(G)\geq k(G/N)$ for any $N\trianglelefteq G$, we may assume that $p$ divides $|\bF_2(G):\bF(G)|$. By the same reason, we may also assume that $\Phi(G)=1$ so that $\bF(G)$ is elementary abelian. By Gasch\"utz's theorem (III.4.2, III.4.4. and III.4.5 of \cite{hup}), there exists $H\leq G$ such that $G=H\bF(G)$ and the action of $H$ on $\bF(G)$ is completely reducible. Write $\bF(G)=V_1\oplus\cdots\oplus V_t$ as a direct sum of irreducible $H$-modules. For every $i$, let $C_i=\bC_H(V_i)$ and note that $k(G)\geq k((H/C_i)V_i)$. Thus, we may assume that $G=HV$, where $V=\bF(G)$ is a faithful irreducible $H$-module in characteristic $q$, for some prime $q\neq p$. Now the result follows from Theorem \ref{prop} and \ref{prop2}.
\end{proof}

It seems that with additional work it might be possible to get $D=1$ in Theorem \ref{prop} when $p$ is large enough. However, it is not clear how to improve the constant in Theorem \ref{prop2} even when $p$ is very large.

We conclude by remarking that, on the other hand,  the stronger bound in Theorem \ref{11000} also holds for small primes. In fact, even Conjectures C and D hold when $p$ is small enough. More precisely, using \cite{vl1, vl3, vl2} we have checked that these conjectures hold when $k(G)\leq14$. It follows that Conjecture D holds when $p<29$, Conjecture C holds when $p<100$ and the stronger bound in Theorem \ref{11000} holds when $p<230$.







\section{Blocks and characters of $p'$-degree}

As we have mentioned in the Introduction,  the Heth\'elyi-K\"ulshammer problem on the number of conjugacy classes has interesting connections with representation theory. For instance, they already pointed out that the inequality $k(G)\geq2\sqrt{p-1}$ holds for every finite group if it holds for $p$-solvable groups and the Alperin-McKay conjecture is true for every finite group (McKay's conjecture is enough; see \cite{kmz} for this and related questions). The straightforward versions of the results in this paper for $p'$-degree irreducible characters and Brauer $p$-blocks are not true. Recall from \cite{hk1} that the groups $H_p=\CCC_{\sqrt{p-1}}\ltimes\CCC_p$ where $p$ is a prime such that $p-1$ is a square have $2\sqrt{p-1}$ conjugacy classes. Let $V_p$ be a faithful irreducible $H_p$-module over $\GF(q)$, where $q\neq p$ is prime,  and let $G_p=H_p\ltimes V_p$. Then $\Irr_{p'}(G_p)=\Irr(H_p)$ so
$$
|\Irr_{p'}(G_p)|=k(H_p)=2\sqrt{p-1},
$$
and there is no hope to get a better lower bound for the number of $p'$-degree irreducible characters of a finite group $G$ such that $p$ divides $|G:\bF(G)|$ than the one in \cite{mm}. Similarly, if $B_0(G_p)$ is the principal $p$-block of $G_p$, then $k(B_0(G_p))=k(H_p)=2\sqrt{p-1}$. However, the following natural versions could have an affirmative answer.

\begin{que}
\label{pprime}
Is it true that if 
 $G$ is a finite group such  that 
$$
\bigcap_{\chi\in\Irr_{p'}(G)}\Ker\chi=\{1\}
$$
and $p$ divides $|G:\bF(G)|$ then $k(G)\geq p/\log p$?
\end{que}

The subgroup $\bigcap_{\chi\in\Irr_{p'}(G)}\Ker\chi$ is known to have interesting properties. For instance, by a theorem of Y. Berkovich (Theorem 7.7 of \cite{nav2}) it has a normal $p$-complement. 

\begin{que}
\label{block}
Is it true that if 
 $B$ is a faithful Brauer $p$-block of positive defect of a finite group $G$ and $p$  divides $|G:\bF(G)|$ then $k(B)\geq p/\log p$?
\end{que}

It would be interesting to see if the proposed bounds in these questions are of the right order of magnitude. In other words, does there exist a universal constant $c>0$ such that $k(G)$ (respectively $k(B)$) is at least $cp/\log p$?

\end{document}